\documentclass[12pt, reqno]{amsart}
\usepackage{amsmath, amsthm, amscd, amsfonts, amssymb, graphicx, color}
\usepackage[bookmarksnumbered, colorlinks, plainpages]{hyperref}
\hypersetup{colorlinks=true,linkcolor=red, anchorcolor=green, citecolor=cyan, urlcolor=red, filecolor=magenta, pdftoolbar=true}
\input{mathrsfs.sty}

\newtheorem{theorem}{Theorem}[section]
\newtheorem{lemma}[theorem]{Lemma}
\newtheorem{proposition}[theorem]{Proposition}
\newtheorem{cor}[theorem]{Corollary}
\theoremstyle{definition}

\theoremstyle{remark}
\newtheorem{remark}[theorem]{\bf{Remark}}
\numberwithin{equation}{section}
\begin{document}

\title [Generalized Euclidean operator radius inequalities $2$-tuple operators] { Refinements of generalized Euclidean operator radius inequalities of $2$-tuple operators }

\author[S. Jana, P. Bhunia, K. Paul] {Suvendu Jana, Pintu Bhunia, Kallol Paul$^*$}

\address{(Jana) Department of Mathematics, Mahishadal Girl's College, Purba Medinipur 721628, West Bengal, India}
\email{janasuva8@gmail.com}
 
\address{(Bhunia) Department of Mathematics, Jadavpur University, Kolkata 700032, West Bengal, India}
\email{pintubhunia5206@gmail.com}
\address{(Paul) Department of Mathematics, Jadavpur University, Kolkata 700032, West Bengal, India}
\email{kalloldada@gmail.com}

\thanks{* \textit{Corresponding author}\\
	\indent Pintu Bhunia would like to thank UGC, Govt. of India for the financial support in the form of Senior Research Fellowship under the mentorship of Prof. Kallol Paul. }

\renewcommand{\subjclassname}{\textup{2020} Mathematics Subject Classification}

\maketitle

\begin{abstract}
 We develop several upper and lower bounds for the $A$-Euclidean operator radius of $2$-tuple operators admitting $A$-adjoint, and show that they refine the earlier related bounds. As an application of the bounds developed here, we obtain  sharper $A$-numerical radius bounds.
\end{abstract}

\section{Introduction}

\noindent Let $\mathscr{H}$ be a complex Hilbert space with inner product $\langle \cdot,\cdot \rangle $ and let $\|\cdot\|$ be the norm induced by the inner product. Let $ \mathbb{B}(\mathscr{H})$ denote the $C^*$-algebra of all bounded linear operators on $\mathscr{H}.$  For $A\in \mathbb{B}(\mathscr{H}),$ $A^*$ denotes the adjoint of $A$, and  $|A|=({A^*A})^{\frac{1}{2}}$.
Also, $\mathcal{R}(A)$ and $\mathcal{N}(A)$ denote the range and the kernel of $A$, respectively.  Every positive operator $A$ in $\mathbb{B}(\mathscr{H})$ defines the following positive semi-definite sesquilinear form: 
$$\langle . , . \rangle_A: \mathscr{H}\times \mathscr{H}\rightarrow \mathbb{C},\hspace{0.4cm} (x,y)\rightarrow \langle x, y\rangle_A=\langle Ax, y\rangle.$$ 
Seminorm $\|\cdot\|_A$ induced by the semi-inner product $\langle . , . \rangle_A$, is given by $\|x\|_A=\langle Ax, x\rangle^{1/2}=\|A^{1/2}x\| .$ This makes $\mathscr{H}$ into  a semi-Hilbertian space. It is easy to verify that the seminorm induces a norm if and only if $A$ is injective. Also, $(\mathscr{H},\| \cdot \|_A)$ is complete if and only if  the range space of operator $A$,  denoted by $\mathcal{R}(A),$ is closed subspace of $\mathscr{H}.$ Henceforth, we reserve  the symbol $ A $  for a non-zero positive operator in $ \mathbb{B}(\mathscr{H})$.   We denote  the $A$-unit sphere and $A$-unit ball of the semi-Hilbertian space $(\mathscr{H}, \|\cdot\|_A)$ by $\mathbb{S}_{{\| \cdot \|}_A}$ and $\mathbb{B}_{{\| \cdot \|}_A}$,  respectively, i.e., $$ \mathbb{S}_{{\| \cdot \|}_A}=\left\lbrace x\in\mathscr{H}: \|x\|_A=1\right\rbrace,  \,\, \mathbb{B}_{{\| \cdot \|}_A}=\left\lbrace x\in\mathscr{H}: \|x\|_A \leq 1\right\rbrace      .$$  
 For $T\in \mathbb{B}(\mathscr{H})$, let $ c_A(T) $ and $w_A(T)$ denote the $A$-Crawford number and the $A$-numerical radius of $T$, respectively and are defined as  $$c_A(T)=\inf \left \{|\langle Tx,x \rangle_A|: x\in \mathbb{S}_{{\| \cdot \|}_A} \right \}, \, \, w_A(T)=\sup \left \{|\langle Tx,x \rangle_A|: x\in \mathbb{S}_{{\| \cdot \|}_A} \right \}.$$ 
 Note that  $w_A(T)$ is not necessarily finite, see \cite{AOT20}. 
  An operator $ S\in\mathbb{B}(\mathscr{H})$ is called an $A$-adjoint of $ T\in\mathbb{B}(\mathscr{H})$ if for every $ x,y\in \mathscr{H}$, $ \langle Tx, y \rangle_A=\langle x, Sy\rangle_A $ holds, i.e.,  $S$ is a solution of the operator equation $ AX=T^*A$.
 There are operators $T$ for which $A$-adjoint may fail to exist, when it do exist then there may be more than one $A$-adjoint. The set of all operators in $\mathbb{B}(\mathscr{H})$ which possess $A$-adjoint is denoted by $\mathbb{B}_A(\mathscr{H}).$ By Douglas theorem \cite{D1}, we have \begin{eqnarray*}
\mathbb{B}_A(\mathscr{H})&=&\left\lbrace T\in\mathbb{B}(\mathscr{H}):\hspace{0.1cm} \mathcal{R}(T^*A)\subseteq \mathcal{R}(A)\right\rbrace \\&=& \left\lbrace  T\in\mathbb{B}(\mathscr{H}):\hspace{0.1cm} \exists  \,\,\lambda > 0\textit{ such that}\hspace{0.2cm}  \|ATx\|\leq \lambda \|Ax\|,\hspace{0.1cm} \forall x\in\mathscr{H}\right\rbrace.
\end{eqnarray*}
If $T\in\mathbb{B}_A(\mathscr{H})$, then there exists a unique solution of $ AX=T^*A$, is denoted by $T^{\sharp_A}$, satisfying $\mathcal{R}(T^{\sharp_A})\subseteq \overline{\mathcal{R}(A)}$, where $\overline{\mathcal{R}(A)}$ is the norm closure of $\mathcal{R}(A)$. 
For simplicity we will write $ T^\sharp$ instead of  $T^{\sharp_A} .$
   If $ T\in\mathbb{B}_A(\mathscr{H}),$ then $ T^\sharp  \in\mathbb{B}_A(\mathscr{H}). $  Moreover, $ \left[T^\sharp\right]^\sharp=P_{\overline{\mathcal{R}(A)}}TP_{\overline{\mathcal{R}(A)}}$ and $ \left[\left[T^\sharp\right]^\sharp\right]^\sharp=T^\sharp $, where $ P_{\overline{\mathcal{R}(A)}}$ denotes the orthogonal projection onto $\overline{\mathcal{R}(A)}$.
   For more about $ T^\sharp$, the reader can see \cite{Ar1,Ar2}. 
  Again, clearly we have 
    \begin{eqnarray*}
    	\mathbb{B}_{A^{1/2}}(\mathscr{H})&=&\left\lbrace T\in\mathbb{B}(\mathscr{H}):\hspace{0.1cm} \mathcal{R}(T^*A^{1/2})\subseteq \mathcal{R}(A^{1/2})\right\rbrace \\&=& \left\lbrace  T\in\mathbb{B}(\mathscr{ H}):\hspace{0.1cm} \exists\,\,  \lambda > 0\textit{ such that}\hspace{0.2cm}  \|Tx\|_A\leq \lambda\|x\|_A,\hspace{0.1cm} \forall x\in\mathscr{H}\right\rbrace.
    \end{eqnarray*}
   An operator in $\mathbb{B}_{A^{1/2}}(\mathscr{H})$  is called $A$-bounded operator. The inclusion $\mathbb{B}_{A}(\mathscr{H}) \subseteq \mathbb{B}_{A^{1/2}}(\mathscr{H}) $ always holds. Both of them  are subalgebras of $\mathbb{B}(\mathscr{H})$ which are neither closed and nor dense in $\mathbb{B}(\mathscr{H}).$ 
   The semi-inner product $\langle  . , . \rangle_A$ induces the $A$-operator seminorm on  $\mathbb{B}_{A^{1/2}}(\mathscr{H})$ defined as follows: 
   \begin{eqnarray*}
   \|T\|_A&=&\sup_{\underset{x\neq 0}{x\in\overline{\mathcal{R}(A)}}}\frac{\|Tx\|_A}{\|x\|_A}=\sup\left\lbrace \|Tx\|_A:\hspace{0.1cm}  x\in \mathbb{S}_{{\| \cdot \|}_A} \right\rbrace< \infty.
\end{eqnarray*}
Also, it is easy to verify that
    \begin{eqnarray*}
    	\|T\|_A&=&\sup\left\lbrace | \langle Tx,y\rangle_A|:\hspace{0.1cm}  x,y \in \mathbb{S}_{{\| \cdot \|}_A}  \right\rbrace.
    \end{eqnarray*}
By Cauchy-Schwarz inequality, it follows that $| \langle Tx,x\rangle_A| \leq \|Tx\|_A \|x\|_A$ for all $x\in \mathscr{ H},$ and so $w_A(T)\leq \|T\|_A$ for all $ T\in \mathbb{B}_{A^{1/2}}(\mathscr{H}) $. For  $A$-selfadjoint operator $T$ (i.e., $AT=T^*A$),  we have $w_A(T)=\|T\|_A,$ see in \cite{Z}. An operator $T\in \mathbb{B}_{A}(\mathscr{H}) $ can be expressed as $T=\Re_A(T)+i \Im_A(T)$, where $\Re_A(T)=\frac12(T+T^{\sharp_A})$ and $\Im_A(T)=\frac1{2i}(T-T^{\sharp_A}).$  This decomposition is called $A$-Cartesian decomposition, using this we have $|\langle \Re_A(T)x,x\rangle_A|^2+ |\langle \Im_A(T)x,x\rangle_A|^2=|\langle Tx,x\rangle_A|^2$ for all $x\in \mathscr{H}.$ This implies $\|\Re_A(T)\|_A \leq w_A(T)$ and $\|\Im_A(T)\|_A \leq w_A(T)$, since $\Re_A(T)$ and $\Im_A(T)$ both are $A$-selfadjoint.  Therefore, $\|T\|_A\leq  \|\Re_A(T)+i \Im_A(T)\|_A \leq 2w_A(T).$ Thus, for every $T\in \mathbb{B}_{A}(\mathscr{H}),$ we get $w_A(T)\leq \|T\|_A\leq 2w_A(T)$. 
One can also easily verify that the above inequality holds for every  $T \in \mathbb{B} _{A^{1/2}}(\mathscr{H}),$ and $w_A(T^n)\leq [w_A(T)]^n$ holds for every positive integer $n$, see \cite{BFM}.
 
Following \cite{Pop}, the $A$-Euclidean operator radius of $d$-tuple operators $\mathbf{T}=(T_1,T_2,.....,T_d)\in{\mathbb{B}_{A^{1/2}}(\mathscr{H})}^d $ is defined as
 $$ w_{A,e}(\mathbf{T})=\sup\left\lbrace \left(\sum_{k=1}^{d} |\langle T_kx,x\rangle_A|^2\right)^{1/2} : x\in \mathbb{S}_{{\| \cdot \|}_A}  \right\rbrace.$$  
 This is also known as $A$-joint numerical radius of $\mathbf{T}$.
 The $A$-Euclidean operator seminorm of $d$-tuple operators $\mathbf{T}=(T_1,T_2,.....,T_d)\in{\mathbb{B}_{A^{1/2}}(\mathscr{H})}^d $ is defined as  $$\|\mathbf{T}\|_A=\sup\left\lbrace \left(\sum_{k=1}^{d} \| T_kx\|_A^2\right)^{1/2} : x\in \mathbb{S}_{{\| \cdot \|}_A} \right\rbrace.$$
 Clearly, the $A$-Euclidean operator radius and $A$-Euclidean operator seminorm of $d$-tuple operators are generalizations of $A$-numerical radius and $A$-operator seminorm of an operator in $\mathbb{B}_{A^{1/2}}(\mathscr{H})$. Observe that for $A=I,$ $\|\cdot \|_A=\|\cdot\|,$ $w_A(\cdot)=w(\cdot)$, $c_A(\cdot)=c(\cdot)$, $w_{A,e}(\cdot)= w_e(\cdot)$ and $\| \cdot \|_{A,e}=\|\cdot \|_e$ are the usual operator norm, numerical radius, Crawford number, Euclidean operator radius and Euclidean operator norm, respectively. 
 For recent developments of $A$-numerical radius inequalities  see \cite{pkk,BNP} and  
for Euclidean operator radius inequalities see \cite{D,MSS,SRS}. 
 In this paper,  we obtain several inequalities involving $A$-Euclidean operator radius and $A$-Euclidean operator seminorm of $2$-tuple operators, and we show that these inequalities improve on the earlier related inequalities.
  
   We end this introductory section with a brief description of the space $\mathbf{R}(A^{1/2})$ ( see \cite{Ar}) as follows: 
  The semi-inner product $\langle . , . \rangle_A$  induces an inner product on the quotient space $\mathscr{H}/\mathcal{N}(A) ,$ defined by $[\overline{x},\overline{y}]=\langle Ax, y\rangle,\,\,\, \forall \,\,\overline{x},\overline{y}\in\mathscr{H}/\mathcal{N}(A) .$ The space $( \mathscr{H}/\mathcal{N}(A), [. , .]) $  is, in general, not a complete space. 
 The completion of $( \mathscr{H}/\mathcal{N}(A), [. , .]) $ is isometrically isomorphic to the Hilbert space ${R}(A^{1/2})$ via the canonical construction mentioned in \cite{BR}, where  ${R}(A^{1/2})$ is equipped  with the inner product $$( A^{1/2} x, A^{1/2} y )=\langle P_{\overline{\mathcal{R}(A)}}x,P_ {\overline{\mathcal{R}(A)} }y\rangle,\,\,\, \forall x,y\in\mathscr{H}.$$ In the sequel, the Hilbert space $( \mathcal{R}(A^{1/2}), (. , . ))$ will be denoted by $\mathbf{R}(A^{1/2})$ and we use the symbol $ \| \cdot \|_{\mathbf{R}(A^{1/2})}$ to represent the norm induced by the inner product $ ( . , . )$.
  Note that, the fact $\mathcal{R}(A)\subseteq \mathcal{R}(A^{1/2})$ implies that $ ( Ax, Ay ) = \langle x, y\rangle_A, \hspace{0.2cm}\forall x,y\in\mathscr{H}.$ This gives $ \|Ax\|_{\mathcal{R}(A^{1/2})}=\|x\|_A ,\hspace{0.2cm} \forall x\in\mathscr{H}.$ Now, we give a nice connection of an operator $T\in \mathbb{B}_{A^{1/2}}(\mathscr{H})$ with an operator $\widetilde{T}\in\mathbb{B}(\mathbf{R}(A^{1/2})),$ in the form of the following proposition, see \cite{Ar}. 
  
   \begin{proposition}  \label{p10} Let $ T\in\mathbb{B}(\mathscr{H})$. Then $ T\in\mathbb{B}_{A^{1/2}}(\mathscr{H})$ if and only if there exist a unique $\widetilde{T}\in\mathbb{B}(\mathbf{R}(A^{1/2}))$ such that $ Z_AT=\widetilde{T}Z_A,$ where $ Z_A:\mathscr{H}\rightarrow \mathbf{R}(A^{1/2})$ is defined by $Z_Ax=Ax $.
\end{proposition}

\section{Main Results}

We begin with  the following sequence of known lemmas. First lemma is known as mixed Schwarz inequality.
 
  \begin{lemma}\cite{a1}
 If $T\in\mathbb{B}(\mathscr{H})$ and $ 0\leq\alpha\leq1$, then $$ |\langle Tx,y\rangle|^2\leq\langle|T|^{2\alpha} x,x\rangle\langle|T^*|^{2(1-\alpha)}y,y\rangle\,\, \forall\,\, x,y\in\mathscr{H}.$$ 
 \label{lem1}\end{lemma}

Second lemma is known as Holder-McCarthy inequality.

\begin{lemma}\cite{a2}
	If $T\in\mathcal{B}(\mathscr{H})$ is positive, then the following inequalities hold: For  any $x\in\mathscr{H}$,
	 $$\langle A^rx,x\rangle\geq||x||^{2(1-r)}\langle Ax,x\rangle^r,\,\,\,\,\, \textit{for}\,\,\,  r\geq1$$
	 and 
	 $$\langle A^rx,x\rangle\leq||x||^{2(1-r)}\langle Ax,x\rangle^r,\,\,\,\,\, \textit{for}\,\,\,  0\leq r\leq1. $$ 
\label{a3}\end{lemma}

Third lemma is related to $A$-selfadjoint operators.

\begin{lemma}\cite{kf2}
Let $ T\in\mathcal{B}(\mathscr{H})$ be $A$-selfadjoint. Then $T^\sharp$ is also $A$-selfadjoint and $  [T^\sharp]^\sharp=T^\sharp. $
\label{refa1}\end{lemma}

Fourth lemma is related to semi-Hilbertian space operator $T$ and Hilbert space operator $\widetilde{T}$.

\begin{lemma}\label{a2}\cite{Ar, kf}     
	Let $T\in \mathcal{B}_A(\mathcal{H})$. Then\\
	$(i)~\widetilde{T^{\sharp}}=\big(\widetilde{T}\big)^*\;\text{ and }\; \widetilde{({T^{\sharp_A}})^{\sharp_A}}=\widetilde{T}.\\
	(ii)~\|T\|_A=\|\widetilde{T}\|_{\mathcal{B}(\mathbf{R}(A^{1/2}))}, ~ w_A(T)=w(\widetilde{T}) ~ \mbox{and}~~c_A(T)=c(\widetilde{T}).$ 
	
	\noindent (Here $\|\widetilde{T}\|_{\mathcal{B}(\mathbf{R}(A^{1/2}))}$ denotes the usual operator norm of $\widetilde{T}$).
\end{lemma}

Now, we prove the following result related to $A$-Euclidean operator radius and Euclidean operator radius.

\begin{theorem} \label{lemm1}
	Let $\mathbf{T}=(T_1,T_2, \ldots, T_d) \in {\mathbb{B}_{A^{1/2}}(\mathscr{H})}^d.$ Then
	\begin{eqnarray*}
		w_{A,e}(\mathbf{T})=w_{A,e}(T_1,T_2,...,T_d)=w_e(\widetilde{T_1},\widetilde{T_2},....,\widetilde{T_d})=w_e(\widetilde{\mathbf{T}})
	\end{eqnarray*} 
	where $\widetilde{\mathbf{T}}= (\widetilde{T_1},\widetilde{T_2},....,\widetilde{T_d})\in {\mathbb{B}(\mathbf{R}(A^{1/2}))}^d.$
\end{theorem} 
\begin{proof}
	First we prove $w_{A,e}(\mathbf{T})\leq w_e(\widetilde{\mathbf{T}})$. We recall that
	\begin{eqnarray*}
		w_{A,e}(\mathbf{T})&=&\sup\left\lbrace\left(\sum_{i=1}^{d} |\langle T_i x,x\rangle|^2\right)^\frac12 :\hspace{0.1cm} x\in\mathscr{H}, \hspace{0.1cm}\|x\|_A=1\right\rbrace\\&=& \sup\left\lbrace\left(\sum_{i=1}^{d} |(A T_i x,Ax)|^2\right)^\frac12 :\hspace{0.1cm} x\in\mathscr{H}, \hspace{0.1cm}\|Ax\|_{\mathbf{R}( A^{1/2})}=1\right\rbrace\\&=& \sup\left\lbrace\left(\sum_{i=1}^{d} |( \widetilde{T_i}A x,Ax)|^2\right)^\frac12 :\hspace{0.1cm} x\in\mathscr{H}, \hspace{0.1cm}\|Ax\|_{\mathbf{R}( A^{1/2})}=1\right\rbrace\\&&\,\,\,\,\,\,\,\,\,\,\,\,\,\,\,\,\,\,\,\,\,\,\,\,\,\,\,\,\,\,\,\,\,\,\,\,\,\,\,\,\,\,\,\,\,\,\,\,\,\,\,\,\,\,\,\,\,\,\,\,\,\,\,\,\,\,\,\,\,\,(\textit{using Proposition \ref{p10}}).
	\end{eqnarray*} 
	From the decomposition $\mathscr{H}=\mathcal{N}(A^{1/2})\oplus\overline{\mathcal{R}(A^{1/2}})$, we obtain that \begin{eqnarray}\label{01}
		w_{A,e}(\mathbf{T})&=& \sup\left\lbrace\left(\sum_{i=1}^{d} |( \widetilde{T_i}A x,Ax)|^2\right)^\frac12 :\hspace{0.1cm} x\in \overline {\mathcal{R}(A^{1/2})}, \hspace{0.1cm}\|Ax\|_{\mathbf{R}( A^{1/2})}=1\right\rbrace.
	\end{eqnarray}
	Now, 
	\begin{eqnarray}\label{02}
	& & 	w_e(\widetilde{\mathbf{T}}) \nonumber\\
		&=&\sup\left\lbrace\left(\sum_{i=1}^{d} |(\widetilde{T_i}y,y)|^2\right)^\frac12:\hspace{0.1cm} y\in\mathcal{R}(A^{1/2}),\hspace{0.1cm} \|y\|_{\mathbf{R}(A^{1/2})}=1\right\rbrace \nonumber\\
		&=&\sup\left\lbrace\left(\sum_{i=1}^{d} |(\widetilde{T_i}A^{1/2}x,A^{1/2}x)|^2\right)^\frac12:\hspace{0.1cm} x\in\mathscr{H},\hspace{0.1cm} \|A^{1/2}x\|_{\mathbf{R}(A^{1/2})}=1\right\rbrace \nonumber\\
		&=&\sup\left\lbrace\left(\sum_{i=1}^{d} |(\widetilde{T_i}A^{1/2}x,A^{1/2}x)|^2\right)^\frac12: x\in\overline{\mathcal{R}(A^{1/2})}, \|A^{1/2}x\|_{\mathbf{R}(A^{1/2})}=1\right\rbrace.
	\end{eqnarray}
	Since $\mathcal{R}(A) \subseteq \mathcal{R}(A^{1/2})$,  \eqref{01} together with  \eqref{02}  implies $w_{A,e}(\mathbf{T}) \leq w_e(\widetilde{\mathbf{T}}).$
	
	Next we show the reverse inequality, i.e, $ w_A(\widetilde{\mathbf{T}})\leq w_{A,e}(\mathbf{T}).$
	Suppose that 
	\begin{eqnarray*}
		\beta\in\left\lbrace\left(\sum_{i=1}^{d} |(\widetilde{T_i}A^{1/2}x,A^{1/2}x)|^2\right)^\frac12:\hspace{0.1cm} x\in\overline{\mathcal{R}(A^{1/2})},\hspace{0.1cm} \|A^{1/2}x\|_{\mathbf{R}(A^{1/2})}=1\right\rbrace =W_{e}(\widetilde{\mathbf{T}}),\, (\textit{say}).
	\end{eqnarray*}
	So, there exists  $ x\in\overline{\mathcal{R}(A^{1/2})}$ with $ \|A^{1/2}x\|_{\mathbf{R}(A^{1/2})}=1$ such that 
	$$\beta=\left(\sum_{i=1}^{d} |(\widetilde{T_i}A^{1/2}x,A^{1/2}x)|^2\right)^\frac12.$$
	Since $ A^{1/2}x \in \mathbf{R}(A^{1/2})$ and $\mathcal{R}(A) $ is dense in $\mathbf{R}(A^{1/2})$, there exist a sequence $\{x_n\}$ in $\mathscr{H}$ such that $\lim_{n\to \infty} \|Ax_n-A^{1/2}x\|_{\mathbf{R}(A^{1/2})}=0$. Hence   $\beta=\lim_{n\rightarrow\infty} \left(\sum_{i=1}^{d} |(\widetilde{T_i}Ax_n,Ax_n)|^2\right)^\frac12$  and $\lim_{n\rightarrow\infty}\|Ax_n\|_{\mathbf{R}(A^{1/2})}=1.$
	Now, let $y_n=\frac{x_n}{\|Ax_n\|_{\mathbf{R}(A^{1/2})}}.$ Then clearly we have, $\beta=\lim_{n\rightarrow\infty} \left(\sum_{i=1}^{d} |(\widetilde{T_i}Ay_n,Ay_n)|^2\right)^\frac12$  and $\|Ay_n\|_{\mathbf{R}(A^{1/2})}=1.$ Therefore, 
	$$\beta\in\overline {\left\lbrace\left(\sum_{i=1}^{n} |( \widetilde{T_i}A x,Ax)|^2\right)^\frac12 :\hspace{0.1cm} x\in \overline {\mathcal{R}(A^{1/2})}, \hspace{0.1cm}\|Ax\|_{\mathbf{R}( A^{1/2})}=1\right\rbrace}=\overline{W_{A,e}(\mathbf{T})},\, (\textit{say}).$$ 
	Hence, $ W_{e}(\widetilde{\mathbf{T}}) \subseteq \overline{W_{A,e}(\mathbf{T})}.$
	This implies $ w_e(\widetilde{\mathbf{T}}) \leq w_{A,e}(\mathbf{T}),$ and this completes the proof.
\end{proof}

Now, we are in a position to prove the bounds of $A$-Euclidean operator radius. In the following theorem we obtain  upper and lower bound for the $A$-Euclidean operator radius of $2$-tuple operators in $\mathbb{B}_A(\mathscr{H})$ involving $A$-numerical radius.

\begin{theorem}\label{th1} Let  $B,C \in\mathbb{B}_A(\mathscr{H})$, then 
	\begin{eqnarray*}
&& \frac{1}{2} w_A(B^2+C^2)+\frac{1}{2}\max \{w_A(B),w_A(C)\}\big| w_A(B+C)-w_A(B-C)\big|\\
&& \leq w_{A,e}^2(B,C)\\
&& \leq\frac{1}{\sqrt{2}}  w_A((B^\sharp B+C^\sharp C)+i(BB^\sharp+CC^\sharp)).
\end{eqnarray*}
\end{theorem}

\begin{proof}
Let $x \in \mathscr{H}$ with $\|x\|_A=1$. Then we have,
 \begin{eqnarray*}
|\left\langle Bx,x\right\rangle_A|^2+|\left\langle Cx,x\right\rangle_A|^2&\geq& \frac{1}{2}\left(|\left\langle Bx,x\right\rangle_A|+|\left\langle Cx,x\right\rangle_A|\right)^2\\&\geq&\frac{1}{2}\left(|\left\langle Bx,x\right\rangle_A \pm \left\langle Cx,x\right\rangle_A|\right)^2\\&=&\frac{1}{2}|\left\langle (B\pm C)x,x\right\rangle_A|^2.
\end{eqnarray*}
Taking supremum over all $x\in \mathscr{H}$, $\|x\|_A=1$, we get 
\begin{eqnarray}
w_{A,e}^2(B,C)\geq\frac{1}{2} w_A^2(B\pm C).
\label{eq1}\end{eqnarray}
Therefore, it follows from the inequalities in (\ref{eq1}) that
 \begin{eqnarray*}
 w_{A,e}^2(B,C)&\geq&\frac{1}{2}\max \{w_A^2(B+C), w_A^2(B-C)\}\\&=& \frac{w_A^2(B+C)+w_A^2(B-C)}{4}+\frac{\big|w_A^2(B+C)-w_A^2(B-C)\big|}{4}\\&\geq& \frac{w_A((B+C)^2)+w_A((B-C)^2)}{4}\\
 && +(w_A(B+C)+w_A(B-C))\frac{\big|w_A(B+C)-w_A(B-C)\big|}{4}\\&\geq& \frac{w_A((B+C)^2+(B-C)^2)}{4}\\
 && +w_A((B+C)+(B-C))\frac{\big|w_A(B+C)-w_A(B-C)\big|}{4}.	 \end{eqnarray*}
Therefore,
\begin{eqnarray}\label{p1}
	w_{A,e}^2(B,C)&\geq&\frac{w_A(B^2+C^2)}{2}+\frac{w_A(B)}{2}\big|w_A(B+C)-w_A(B-C)\big|.
\end{eqnarray}
Interchanging $B$ and $C$ in \eqref{p1}, we arrive 
 \begin{eqnarray}\label{p2} w_{A,e}^2(B,C)&\geq&\frac{w_A(B^2+C^2)}{2}+\frac{w_A(C)}{2}\big|w_A(B+C)-w_A(B-C)\big|.
\end{eqnarray} 
The inequality \eqref{p1} together with \eqref{p2}, gives the first inequality.

 Next, we prove the second inequality.
 Let $x\in\mathscr{H}$ with $\|x\|=1$. Then we have, 
 \begin{eqnarray*}
&& (|\left\langle Bx,x\right\rangle|^2+|\left\langle Cx,x\right\rangle|^2)^2\\
&&\leq (\left\langle |B|x,x\right\rangle\left\langle |B^*|x,x\right\rangle+\left\langle |C|x,x\right\rangle\left\langle |C^*|x,x\right\rangle)^2\,\,\,(\textit{using Lemma \ref{lem1}})\\
&&\leq  (\left\langle |B|x,x\right\rangle^2+\left\langle |C|x,x\right\rangle^2)(\left\langle |B^*|x,x\right\rangle^2+\left\langle |C^*|x,x\right\rangle^2)\\
  	&&\,\,\,\,\,\,\,\,\,\,\,\,\,\,\,\,\,\,\,\,\,\,\,\, (\textit{since $(ab+cd)^2\leq (a^2+c^2)(b^2+d^2)$ for all  $a,b,c,d\in \mathbb{R}$})\\
  	&&\leq  (\left\langle |B|^2 x,x\right\rangle+\left\langle |C|^2 x,x\right\rangle)(\left\langle |B^*|^2 x,x\right\rangle+\left\langle |C^*|^2 x,x\right\rangle)\,\,\,(\textit{using Lemma \ref{a3}})\\
  	&&= \left\langle (B^*B+C^*C) x,x\right\rangle \left\langle(BB^*+CC^*) x,x\right\rangle\\
  	&&\leq \frac{1}{2}\left\lbrace\left\langle (B^*B+C^*C) x,x\right\rangle^2 +\left\langle(BB^*+CC^*) x,x\right\rangle^2\right\rbrace\\
  	&&=  \frac{1}{2} |\left\langle (B^*B+C^*C) x,x\right\rangle +i\left\langle(BB^*+CC^*) x,x\right\rangle|^2\\
  	&&=\frac{1}{2} |\left\langle ((B^*B+C^*C) +i(BB^*+CC^*)) x,x\right\rangle|^2 \\
  	&&\leq \frac{1}{2} w^2((B^*B+C^*C) +i(BB^*+CC^*)).
 \end{eqnarray*}
 Taking supremum over all $x \in \mathscr{H}$ with $\|x\|=1$, we get
\begin{eqnarray}\label{p3}
	w_e^2(B,C)&\leq&\frac{1}{\sqrt{2}} w((B^*B+C^*C) +i(BB^*+CC^*)).
\end{eqnarray}
As $B,C\in \mathbb{B}_{A^{1/2}}(\mathscr{H})$, following Proposition \ref{p10}, there exist unique $\widetilde{B}$ and $\widetilde{C}$ in $\mathbb{B}(\mathbf{R}(A^{1/2}))$
 such that $ Z_AB=\widetilde{B}Z_A $ and  $ Z_AC=\widetilde{C}Z_A $. 
 The inequality (\ref{p3}) implies that
  \begin{eqnarray}
w_e^2(\widetilde{B},\widetilde{C})\leq\frac{1}{\sqrt{2}}w((\widetilde{B}^*\widetilde{B}+\widetilde{C}^*\widetilde{C})+i (\widetilde{B}\widetilde{B}^*+\widetilde{C}\widetilde{C}^*)).
\label{p4}\end{eqnarray}
Since $ (\widetilde{B})^*=\widetilde{B^\sharp }$, the inequality (\ref{p4}) becomes  \begin{eqnarray}
w_e^2(\widetilde{B},\widetilde{C})\leq\frac{1}{\sqrt{2}}w((\widetilde{B^\sharp}\widetilde{B}+\widetilde{C^\sharp}\widetilde{C})+i (\widetilde{B}\widetilde{B^\sharp}+\widetilde{C}\widetilde{C^\sharp})).
\label{p5}\end{eqnarray}
For any $ S,T\in\mathbb{B}_{A^{1/2}}(\mathscr{H})$, it is easy to see that   $\widetilde{ST}=\widetilde{S}\widetilde{T}$ and $\widetilde{S+\lambda T}=\widetilde{ S}+\lambda \widetilde{T}$ for all $\lambda\in \mathbb{C}$.
 So, the inequality (\ref{p5}) is of the following form
  \begin{eqnarray}
w_e^2(\widetilde{B},\widetilde{C})\leq\frac{1}{\sqrt{2}}w((\widetilde{B^\sharp B+C^\sharp C)+i (B B^\sharp+C C^\sharp)}).
\label{pa6}\end{eqnarray} 
Now, by applying Theorem \ref{lemm1} and Lemma \ref{a2}, we have  $$w_{A,e}^2(B,C)\leq\frac{1}{\sqrt{2}}  w_A((B^\sharp B+C^\sharp C)+i(BB^\sharp+CC^\sharp)).$$
This completes the proof.
 \end{proof}


\begin{remark}
(i) The lower bound of $w_e(B,C)$ in Theorem \ref{th1} is stronger than the lower bound in \cite[Th. 2.8]{kf1}, namely, $\frac{1}{2} w_A(B^2+C^2)\leq w^2_{A,e}(B,C).$
Also, it is not difficult to verify that $$ \frac{1}{\sqrt{2}}w_A((B^\sharp B+C^\sharp C)+i (B B^\sharp+C C^\sharp))\leq \frac{1}{\sqrt{2}} \left\lbrace\|B^\sharp B+C^\sharp C\|^2_A+\|BB^\sharp +CC^\sharp \|^2_A\right\rbrace^\frac12. $$
Therefore, the upper bound of $w_{A,e}(B,C)$ in Theorem \ref{th1} is better than the upper bound in \cite[Th. 2.8]{kf1}, namely, $w^2_{A,e}(B,C)\leq \|BB^\sharp +CC^\sharp \|_A$ if $\|BB^\sharp +CC^\sharp \|_A\leq \|B^\sharp B+C^\sharp C\|_A$. 

(ii) Following Theorem \ref{th1},  $w^2_{A,e}(B,C)=\frac12 w_A(B^2+C^2)$ implies $w_A(B+C)=w_A(B-C).$ However, the converse is not true, in general.

\end{remark}
 
 The following corollary is an immediate consequence of Theorem \ref{th1}.
 
\begin{cor}
	If $ B,C\in\mathbb{B}_A(\mathscr{H})$ are $A$-selfadjoint, then
\begin{eqnarray*}
\frac{1}{2} \|B^2+C^2\|_A+\frac{1}{2}\max \{\|B\|_A,\|C\|_A\}\big| \|B+C\|_A-\|B-C\|_A\big|  \leq w_{A,e}^2(B,C).
\label{eq2}
\end{eqnarray*}
\end{cor}

In particular, considering $ B=[\Re_A(T)]^\sharp $ and $ C=[\Im_A(T)]^\sharp $  in Theorem \ref{th1}, and the using the Lemma \ref{refa1}. we obtain the following new upper and lower bounds for the $A$-numerical radius of a bounded linear operator $T \in \mathbb{B}_A(\mathscr{H})$.
  
\begin{cor} \label{pcor}
	If $T\in \mathbb{B}_A(\mathscr{H})$, then
\begin{eqnarray*}
\frac{1}{4} \|T^\sharp T+TT^\sharp\|_A+ \frac{\alpha}{2}\max \{\|\Re_A(T)\|_A,\|\Im_A(T)\|_A\}
\leq w_A^2(T) \leq \frac{1}{2} \| TT^\sharp+T^\sharp T  \|_A,
\end{eqnarray*}
where $\alpha= \big| \|\Re_A(T)+\Im_A(T)\|_A-\|\Re_A(T)-\Im_A(T)\|_A\big|.$
\end{cor}

 Again, considering $B= T$ and $C= T^\sharp$ in Theorem \ref{th1}, we get the following new lower bound for the $A$-numerical radius of  $T \in \mathbb{B}_A(\mathscr{H})$.
 
\begin{cor}\label{cor1}
	Let $T\in \mathbb{B}_A(\mathscr{H}),$ then
 	\begin{eqnarray*}
		\frac{1}{2}\|\Re_A(T^2)\|_A+\frac{1}{2}w_A(T)\big| \|\Re_A(T)\|_A-\|\Im_A(T)\|_A\big| \leq w_A^2(T).
	\end{eqnarray*}
\end{cor}

To prove our next theorem, we need the following lemma, known as Bohr's inequality.
\begin{lemma} \cite{V}.
	Suppose $a_i\geq 0$ for $ i=1,2,......,n.$ Then $$ \left(\sum _{i=1}^{k} a_i\right)^r=k^{r-1}\sum_{i=1}^{k} a_i^r \,\,\, \textit{for} \,\,\,r\geq1.$$ 
	\label{b2}\end{lemma}

\begin{theorem}
	If $ B,C\in\mathbb{B}_A(\mathscr{H}),$ then $$\frac{1}{8}\|B+C\|_A^4\leq  w_{A,e}(B^\sharp B,C^\sharp C)w_{A,e}(BB^\sharp,CC^\sharp).$$
\end{theorem}
\begin{proof}
	Let $x,y \in \mathscr{H}$ with $\|x\|=\|y\|=1$. Then we have,
	\begin{eqnarray*}
		&& |\langle(B+C)x,y\rangle|^4\\ &=&|\langle Bx,y\rangle+\langle Cx,y\rangle|^4\\&\leq& (|\langle Bx,y\rangle|+|\langle Cx,y\rangle|)^4\\&\leq& 8(|\langle Bx,y\rangle|^4+|\langle Cx,y\rangle|^4)\,\,(\textit{using Lemma \ref{b2})}\\&\leq& 8(\langle|B|x,x\rangle^2\langle |B^*|y,y \rangle^2+\langle|C|)x,x\rangle^2\langle|C^*|y,y\rangle^2)\,\,(\textit{using Lemma \ref{lem1} })\\ &\leq& 8(\langle B^*B x,x\rangle\langle BB^* y,y \rangle+\langle C^*C x,x\rangle \langle CC^* y,y\rangle)\,\,(\textit{using Lemma \ref{a3} })\\ &\leq& 8(\langle B^*B x,x\rangle^2+\langle C^*C x,x\rangle^2)^\frac12(\langle BB^* y,y \rangle^2+\langle CC^* y,y\rangle^2)^\frac12\\
		&\leq& 8 w_e(B^*B,C^*C)w_e(BB^*,CC^*).
	\end{eqnarray*}
	Taking supremum over $\|x\|=\|y\|=1$, we get \begin{eqnarray}
		\frac18\|B+C\|^4\leq  w_e(B^*B,C^*C)w_e(BB^*,CC^*).\label{eqn19}\end{eqnarray}
	As $B,C\in \mathbb{B}_{A^{1/2}}(\mathscr{H})$, following Proposition \ref{p10}, there exist unique $\widetilde{B}$ and $\widetilde{C}$ in $\mathbb{B}(\mathbf{R}(A^{1/2}))$
	such that $ Z_AB=\widetilde{B}Z_A $ and  $ Z_AC=\widetilde{C}Z_A $. 
	The inequality (\ref{eqn19}) implies that\begin{eqnarray}
		\frac18\|\widetilde{B}+\widetilde{C}\|_{\mathbb{B}(\mathbf{R(A^{1/2}))}}^4\leq  w_e(\widetilde{B}^*\widetilde{B},\widetilde{C}^*\widetilde{C})w_e(\widetilde{B}\widetilde{B}^*,\widetilde{C}\widetilde{C}^*).\label{eqn20}
	\end{eqnarray}
	Since $ (\widetilde{B})^*=\widetilde{B^\sharp }$,  the inequality (\ref{eqn20}) becomes
	\begin{eqnarray}
		\frac18\|\widetilde{B}+\widetilde{C}\|_{\mathbb{B}(\mathbf{R(A^{1/2}))}}^4\leq  w_e(\widetilde{B}^\sharp \widetilde{B},\widetilde{C}^\sharp \widetilde{C})w_e(\widetilde{B}\widetilde{B}^\sharp,\widetilde{C}\widetilde{C}^\sharp),\label{eqn21}
	\end{eqnarray}
that is,
	\begin{eqnarray}
		\frac18\|\widetilde{B+C}\|_{\mathbb{B}(\mathbf{R(A^{1/2}))}}^4\leq  w_e(\widetilde{B^\sharp B},\widetilde{C^\sharp C})w_e(\widetilde{BB^\sharp},\widetilde{CC^\sharp}).\label{eqn22}
	\end{eqnarray}
By using Lemma \ref{a2} and Theorem \ref{lemm1} in the above inequality (\ref{eqn22}), we obtain
$$ \frac{1}{8}\|B+C\|_A^4\leq  w_{A,e}(B^\sharp B,C^\sharp C)w_{A,e}(BB^\sharp,CC^\sharp),$$ as desired.
\end{proof}

 Next we obtain an upper bound for the $A$-Euclidean operator radius of $2$-tuple operators admitting $A$-adjoint. First we need the following proposition.

\begin{proposition}
Let $x\in\mathscr{H}$ with $\|x\|_A=1.$ 
	Suppose that  $ T=x\otimes_A x $, is defined as $Tz= (x\otimes x)z=\langle z,x\rangle_A x$ ,\hspace{0.1cm} $\forall z\in\mathscr{H}.$ Then we have
	 $$ |\alpha-1|\leq \|\alpha T-I\|_A\leq \max \left\lbrace 1, |\alpha-1|\right\rbrace, $$ for all $\alpha\in\mathbb{C}.$ Moreover, if $|\alpha-1|\geq 1$, then 
	 $\|\alpha T-I\|_A= |\alpha-1|.$
\label{lemma1}
\end{proposition}

\begin{proof}
For any $z\in \mathscr{H},$ we have 
\begin{eqnarray}
\nonumber\| (\alpha T - I)z\|_A^2&=& \langle (\alpha T - I)z, ˘(\alpha T - I)z \rangle _A\\&=&\nonumber |\alpha|^2\|Tz\|_A^2 -\alpha \langle Tz, z\rangle_A -\bar{\alpha} \langle z, Tz \rangle_A+ \|z\|_A^2\\&=& \nonumber |\langle z, x \rangle_A|^2( |\alpha|^2-\alpha -\bar{\alpha})+ \|z\|_A^2\\&=& |\langle z, x \rangle_A|^2( |\alpha-1|^2-1)+ \|z\|_A^2 \label{equation1}\\
&\leq & \max \{1,|\alpha-1|^2 \} \|z\|_A^2. 
\end{eqnarray}
Taking supremum over $\|z\|_A=1,$ we have
$$ \| \alpha T - I \|_A  \leq  \max \{1,|\alpha-1|\}.$$
Again, from the equation  (\ref{equation1}) we have,
$$\|(\alpha T - I)z\|_A^2 + |\langle z, x \rangle_A|^2= |\langle z, x \rangle_A|^2 |\alpha-1|^2+  \|z\|_A^2.$$
This implies that
$$\|(\alpha T - I)z\|_A \geq  |\langle z, x \rangle_A| |\alpha-1|.$$
 Taking supremum over $\|z\|_A=1,$ we get
 $$ \|\alpha T - I\|\geq \sup_{\|z\|_A=1}|\langle z, x \rangle_A| |\alpha-1| \geq |\alpha-1| .$$
 This completes the proof. 
 \end{proof} 

By using the above proposition we obtain a generalization of Buzano's inequality (\cite{Buz}), in the setting of a  semi-Hilbertian space.
 
\begin{lemma}
If $ x,y,e \in\mathscr{H} $ with $\|e\|_A=1 $, then $$ \mid\langle x,e\rangle_A\langle e,y\rangle_A\mid\leq\frac{\mid\langle x,y\rangle_A \mid +\max\{1,|\alpha -1|\}\|x\|_A\|y\|_A}{|\alpha|},$$
for all non-zero scalar $\alpha.$
\label{lem11}
\end{lemma}

\begin{proof}
	Suppose that $T=e\otimes_A e $. Then we have,
	\begin{eqnarray*}
\mid \alpha \langle x,e\rangle _A \langle e, y\rangle_A-\langle x, y\rangle_A\mid &=& \mid \alpha \langle Tx, y\rangle_A - \langle x, y\rangle_A\mid \\ &=& \mid \langle (\alpha T-I)x,y\rangle_A\mid \\ &\leq & \| \alpha T-I\|_A \|x\|_A\|y\|_A\\&\leq& \max \{ 1,|\alpha-1|\}\|x\|_A\|y\|_A\,\, \textit{(by Proposition \ref{lemma1})}. 
\end{eqnarray*}
This gives that
 \begin{eqnarray*}
\mid \alpha \langle x,e\rangle _A \langle e, y\rangle_A\mid\leq \{ 1, |\alpha-1|\}\|x\|_A\|y\|_A+|\langle  x,y\rangle_A|.
\end{eqnarray*}
This completes the proof. 
\end{proof}

Note that the inequality in Lemma \ref{lem11} was studied (for the case $A=I$) in \cite[Cor. 2.5]{MKD}, using different approaches.
In particular, for $\alpha=2$ in Lemma \ref{lem11}, we have 
\begin{eqnarray}
	\mid  \langle x,e\rangle _A \langle e, y\rangle_A\mid\leq \frac{\|x\|_A\|y\|_A+|\langle  x,y\rangle_A|}{2},
\end{eqnarray}
which was  also obtained in \cite{OAM2021}.


Now, by using Lemma \ref{lem11} we obtain the following upper bound for $A$-Euclidean operator radius.

\begin{theorem}
If $ B,C\in\mathbb{B}_A(\mathscr{H})$, then $$ w_{A,e}^2(B,C)\leq\frac{\max\{ 1,|1-\alpha| \} \|(B,C)\|_{A,e} \|(B^\sharp,C^\sharp)\|_{A,e} +w_A(B^2)+w_A(C^2)}{|\alpha|},$$
for any non-zero scalar $\alpha.$ 
\label{theorem1}\end{theorem}

\begin{proof}
Let $x\in \mathscr{H}$ with $\|x\|_A=1$. Then we have,
\begin{eqnarray*}
&&|\left\langle  Bx,x\right\rangle_A|^2+|\left\langle Cx,x\right\rangle_A|^2\\
&=& |\langle Bx,x \rangle_A\langle x,B^\sharp x \rangle_A|+|\langle Cx,x \rangle_A\langle x,C^\sharp x \rangle_A|\\
&\leq& \frac{\max\{ 1, |\alpha-1|\} \|Bx\|_A\|B^\sharp x\|_A + |\langle Bx, B^\sharp x\rangle_A|}{|\alpha|}\\
&& +\frac{\max\{ 1, |\alpha-1|\} \|Cx\|_A\|C^\sharp x\|_A + |\langle Cx, C^\sharp x\rangle_A|}{|\alpha|}\,\,(\textit{using Lemma \ref{lem11}})\\
&=& \frac{\max\{ 1, |\alpha-1|\}( \|Bx\|_A\|B^\sharp x\|_A+\|Cx\|_A\|C^\sharp x\|_A)}{|\alpha|}\\
&& +\frac{|\langle Bx, B^\sharp x\rangle_A|+|\langle Cx, C^\sharp x\rangle_A|}{|\alpha|}\\
&\leq& \frac{\max\{ 1, |\alpha-1|\}( \|Bx\|^2_A+\|Cx\|^2_A)^\frac{1}{2}(\|B^\sharp x\|^2_A+\|C^\sharp x\|^2_A)^\frac{1}{2}}{|\alpha|}\\
&&+\frac{|\langle B^2x, x\rangle_A|+|\langle C^2x,  x\rangle_A|}{|\alpha|}\\
&\leq& \frac{\max\{ 1, |\alpha-1|\}\|(B,C)\|_{A,e}\|(B^\sharp ,C^\sharp)\|_{A,e}}{|\alpha|}+\frac{w_A(B^2)+w_A(C^2)}{|\alpha|}.
\end{eqnarray*}
Taking supremum over all $x\in \mathscr{H}$ with $\|x\|_A=1$, we get the desired inequality.
\end{proof}
 
 In particular, considering $ B=C=T $ in  Theorem \ref{theorem1}, we obtain the following corollary.
 
\begin{cor}
If $ T\in\mathbb{B}_A(\mathscr{H}) $, then
 $$ w_A^2(T)\leq \frac{\max\{1,|1-\alpha|\}\|T\|^2_A+w_A(T^2)}{|\alpha|},$$
 for any non-zero scalar $\alpha$.
\end{cor}

For $\alpha=2$,  $$ w_A^2(T)\leq \frac{1}{2} \left(\|T\|^2_A+w_A(T^2)\right),$$ which was also obtained in \cite[Cor. 2.5]{kf1}.

Next bound reads as follows:

\begin{theorem}
If $ B,C\in\mathbb{B}_A(\mathscr{H}),$  then 
\begin{eqnarray*}
	w^2_{A,e}(B,C) &\leq& \min\{ w_A^2(B-C),w_A^2(B+C)\}\\
	&&+\frac{\max\{ 1, |1-\alpha|\}\|C^\sharp C+BB^\sharp\|_A+2 w_A(BC)}{|\alpha|},
\end{eqnarray*} for any non-zero scalar $\alpha.$ 
\label{th3}\end{theorem}

\begin{proof}
Let $x\in \mathscr{H}$ with $\|x\|_A=1$. Then we have,
\begin{eqnarray*}
|\left\langle Cx,x\right\rangle_A|^2-2Re{[\left\langle Cx,x\right\rangle_A\overline{\left\langle Bx,x\right\rangle_A}]}+|\left\langle Bx,x\right\rangle_A|^2&=&|\left\langle Cx,x\right\rangle_A-\left\langle Bx,x\right\rangle_A|^2\\&=&|\left\langle (C- B)x,x\right\rangle_A|^2\\&\leq& w_A^2(C-B).
\end{eqnarray*}
Thus,
\begin{eqnarray*}
&& |\left\langle Cx,x\right\rangle_A|^2+|\left\langle Bx,x\right\rangle_A|^2\\
&\leq& w_A^2(C-B)+2Re{[\left\langle Cx,x\right\rangle_A\overline{\left\langle Bx,x\right\rangle_A}]}\\&\leq&w_A^2(C-B)+2|\left\langle Cx,x\right\rangle_A\left\langle Bx,x\right\rangle_A|\\&\leq&w_A^2(C-B)+ \frac{ 2\max\{1,|\alpha-1|\}\|Cx\|_A\|B^\sharp x\|_A+2 |\langle Cx,B^\sharp x\rangle_A|}{|\alpha|} \,\,(\textit{by Lemma \ref{lem11}})\\
&\leq& w_A^2(C-B)+\frac{\max\{ 1, |1-\alpha|\}(\|Cx\|_A^2+\|B^\sharp x\|_A^2)+2w_A(BC)}{|\alpha|}\\&\leq& w_A^2(C-B)+\frac{\max\{ 1, |1-\alpha|\}\|C^\sharp C+BB^\sharp\|_A+2w_A(BC)}{|\alpha|}. 
\end{eqnarray*}
Taking supremum over all $x\in \mathscr{H}$ with $\|x\|_A=1$, we get 
\begin{eqnarray}\label{p6}
	w^2_{A,e}(B,C) \leq w_A^2(B-C)+\frac{\max\{ 1, |1-\alpha|\}\|C^\sharp C+BB^\sharp\|_A+2w_A(BC)}{|\alpha|}. 
\end{eqnarray}
Replacing $C$ by $-C$, we obtain that 
\begin{eqnarray}\label{p7}
	w^2_{A,e}(B,C) \leq w_A^2(B+C)+\frac{\max\{ 1, |1-\alpha|\}\|C^\sharp C+BB^\sharp\|_A+2 w_A(BC)}{|\alpha|}. 
\end{eqnarray}
Following the inequality \eqref{p7} together with \eqref{p6}, we get the desired inequality.
\end{proof}

In particular, considering $\alpha=2$ in Theorem \ref{th3}, we get 
\begin{eqnarray}
w^2_{A,e}(B,C) \leq\min\{ w_A^2(B-C), w_A^2(B+C)\}+\frac{ \|C^\sharp C+BB^\sharp\|_A+ 2w_A(BC)}{2}.\label{p8}\end{eqnarray}

Again, considering $B=C=T$ in   Theorem \ref{th3}, we get the following upper bound 
for the $A$-numerical radius of $T\in \mathbb{B}_A(\mathscr{H})$: 
\begin{eqnarray}
	 w_A^2(T)\leq\frac{\frac{1}{2}\max\{ 1, |1-\alpha|\}\|T^\sharp T+TT^\sharp\|_A+w_A(T^2)}{|\alpha|}.
	 \label{p80}\end{eqnarray}
Putting $\alpha=2$ in   \eqref{p80}, we get 
$$ w_A^2(T)\leq \frac{1}{4}\|T^\sharp T+TT^\sharp\|_A+\frac{1}{2}w_A(T^2),$$ 
which was also  obtained in \cite[Th. 2.11]{Z}.


Next, in the following theorem we obtain a lower bound for $w_{A,e}(B,C)$.

\begin{theorem}
	If $ B,C\in\mathbb{B}_A(\mathscr{H})$, then 
	$$\frac{1}{2}\max\left\lbrace w_A^2(B+C)+c_A^2(B-C),w_A^2(B-C)+c_A^2(B+C)\right\rbrace\leq w_{A,e}^2(B,C).$$
	\label{th2}\end{theorem}
\begin{proof}
	Let $x \in \mathscr{H}$ with $\|x\|_A=1$. Then we have,  
	\begin{eqnarray*}|\left\langle Bx,x\right\rangle_A+\left\langle Cx,x\right\rangle_A|^2+|\left\langle Bx,x\right\rangle_A-\left\langle Cx,x\right\rangle_A|^2=2(|\left\langle Bx,x\right\rangle_A|^2+|\left\langle Cx,x\right\rangle_A|^2).\end{eqnarray*}
	This implies that 
	\begin{eqnarray*}|\left\langle (B+C)x,x\right\rangle_A|^2+|\left\langle (B-C)x,x\right\rangle_A|^2&=&2(|\left\langle Bx,x\right\rangle_A|^2+|\left\langle Cx,x\right\rangle_A|^2)\\&\leq& 2w_{A,e}^2(B,C).
	\end{eqnarray*}
	Thus,
	\begin{eqnarray*}
		|\left\langle (B+C)x,x\right\rangle_A|^2&\leq&2w_{A,e}^2(B,C)-|\left\langle (B-C)x,x\right\rangle_A|^2\\&\leq&2w_{A,e}^2(B,C)-c_A^2(B-C). 
	\end{eqnarray*}
	Taking supremum over all $x \in \mathscr{H}$ with $\|x\|_A=1$, we get $$ w_A^2(B+C)\leq2w_{A,e}^2(B,C)-c_A^2(B-C),$$ that is, \begin{eqnarray} \label{eq3}
		w_A^2(B+C)+c_A^2(B-C)\leq2w_{A,e}^2(B,C).
	\end{eqnarray}
	Similarly, 
	\begin{eqnarray}\label{eq4}
		w_A^2(B-C)+c_A^2(B+C)\leq2w_{A,e}^2(B,C).
	\end{eqnarray} 
	Combining the inequalities (\ref{eq3})  and (\ref{eq4}) we obtain   $$\frac{1}{2}\max\left\lbrace w_A^2(B+C)+c_A^2(B-C),w_A^2(B-C)+c_A^2(B+C)\right\rbrace\leq w_{A,e}^2(B,C),$$ as desired.
\end{proof}

Note that, for $A$-selfadjoint operators $B$ and $C $, the bound in Theorem \ref{th2} is of the form \begin{eqnarray}
	\frac{1}{2}\max\left\lbrace \|B+C\|_A^2+c_A^2(B-C),\|B-C\|_A^2+c_A^2(B+C)\right\rbrace\leq w_{A,e}^2(B,C).
\end{eqnarray}

Also observe that the bound obtained in Theorem \ref{th2} is stronger then the first bound in \cite[Th. 2.7]{kf1}. Next inequality reads as follows:


\begin{theorem}\label{cor2}
	If $ B,C\in\mathbb{B}_A(\mathscr{H})$, then $$\max\left\lbrace w_A^2(B)+c_A^2(C),w_A^2(C)+c_A^2(B)\right\rbrace\leq w_{A,e}^2(B,C).$$
\end{theorem}
\begin{proof}
	Let $x\in \mathscr{H}$ with $\|x\|_A=1$. Then we have,
	\begin{eqnarray*}|\left\langle Bx,x\right\rangle_A+\left\langle Cx,x\right\rangle_A|^2+|\left\langle Bx,x\right\rangle_A-\left\langle Cx,x\right\rangle_A|^2=2(|\left\langle Bx,x\right\rangle_A|^2+|\left\langle Cx,x\right\rangle_A|^2),\end{eqnarray*}
	that is,
	\begin{eqnarray*}|\left\langle (B+C)x,x\right\rangle_A|^2+|\left\langle (B-C)x,x\right\rangle_A|^2=2(|\left\langle Bx,x\right\rangle_A|^2+|\left\langle Cx,x\right\rangle_A|^2).\end{eqnarray*}
	This implies that \begin{eqnarray}
		w_{A,e}^2(B+C,B-C)=2w_{A,e}^2(B,C).
		\label{eq5}\end{eqnarray} 
	Now, replacing $B$ by $B+C$ and $C$ by $B-C$ in Theorem \ref{th2}, we obtain
	\begin{eqnarray}
		2\max\left\lbrace w_A^2(B)+c_A^2(C),w_A^2(C)+c_A^2(B)\right\rbrace\leq w_{A,e}^2(B+C,B-C).
		\label{eq6}\end{eqnarray} 
	The desired inequality follows from (\ref{eq6}) together with the equality (\ref{eq5}).
	
\end{proof}

Finally, we obtain the following upper and lower bounds for $A$-Euclidean operator radius involving $A$-numerical radius.

\begin{theorem}\label{th4}
Let $ B,C\in\mathbb{B}(\mathscr{H})$, then 
\begin{eqnarray*}\label{pmax}
	w^2_A(\sqrt{\alpha}B\pm\sqrt{1-\alpha}C) \leq w_{A,e}^2(B,C) \leq  w_A^2(\sqrt{\alpha} B+\sqrt{1-\alpha}C)+w_A^2(\sqrt{1-\alpha} B+\sqrt{\alpha}C), 
	\end{eqnarray*}
 for all $\alpha\in [0,1].$
\end{theorem}

\begin{proof}
Let $x\in \mathscr{H}$ with $\|x\|_A=1$. Then  we have,
\begin{eqnarray*}
&& \sqrt{\alpha}|\langle Bx,x\rangle_A|+\sqrt{1-\alpha}|\langle Cx,x\rangle_A|\\
&& \leq(|\langle Bx,x\rangle_A|^2+|\langle Cx,x\rangle_A|^2)^{\frac{1}{2}}((\sqrt{\alpha})^2+(\sqrt{1-\alpha})^2)^{\frac{1}{2}}\\
  	&&=(|\langle Bx,x\rangle_A|^2+|\langle Cx,x\rangle_A|^2)^{\frac{1}{2}}.
\end{eqnarray*}
Therefore,
\begin{eqnarray*}
(|\langle Bx,x\rangle_A|^2+|\langle Cx,x\rangle_A|^2)^{\frac{1}{2}}&\geq&|\langle \sqrt{\alpha}Bx,x\rangle_A|+|\langle\sqrt{1-\alpha} Cx,x\rangle_A|\\&\geq&|\langle \sqrt{\alpha}Bx,x\rangle_A\pm\langle\sqrt{1-\alpha} Cx,x\rangle_A|\\&=&|\langle\left( \sqrt{\alpha}B\pm\sqrt{1-\alpha} C\right) x,x\rangle_A|.
\end{eqnarray*}
Taking supremum over all $x$ in $\mathscr{H}$ with $\|x\|_A=1$, we get the first inequality, i.e., $$ w_{A,e}(B,C)\geq w_A(\sqrt{\alpha}B\pm\sqrt{1-\alpha}C).$$ 
Next, we prove the second inequality. By simple calculation, we get
\begin{eqnarray*}
&&	|\langle Bx,x\rangle_A|^2+|\langle Cx,x\rangle_A|^2\\
&=&|\langle \sqrt{\alpha}Bx,x\rangle_A+\langle\sqrt{1-\alpha}Cx,x\rangle_A|^2+|\langle \sqrt{1-\alpha}Bx,x\rangle_A-\langle\sqrt{\alpha}Cx,x\rangle_A|^2\\&=&|\langle (\sqrt{\alpha}B+\sqrt{1-\alpha}C)x,x\rangle_A|^2+|\langle( \sqrt{1-\alpha}B-\sqrt{\alpha}C)x,x\rangle_A|^2\\&\leq& w_A^2(\sqrt{\alpha}B+\sqrt{1-\alpha}C)+w_A^2( \sqrt{1-\alpha}B-\sqrt{\alpha}C).
\end{eqnarray*}
Taking supremum over all $x$ in $\mathscr{H}$ with $\|x\|_A=1$, we get $$ w_{A,e}^2(B,C)\leq w_A^2(\sqrt{\alpha} B+\sqrt{1-\alpha}C)+w_A^2(\sqrt{1-\alpha} B-\sqrt{\alpha}C),$$ as desired.
\end{proof}

\begin{remark}
(i) It is easy to verify that 
\begin{eqnarray*}
	w^2_{A,e}(B,C) &\geq & \max_{0\leq \alpha\leq 1} w^2_A(\sqrt{\alpha}B\pm\sqrt{1-\alpha}C) \\
	&\geq & \frac12 \max w^2_A(B\pm C)\\
	&\geq & \frac12 w_A(B^2+C^2).
\end{eqnarray*} 	
	
(ii)  Putting $B=\Re_A(T)$ and $C=\Im_A(T)$ in (i) we obtain that 
\begin{eqnarray*}
	 w_A^2(T) &\geq& \frac12 \max \left\|\Re_A(T) \pm \Im_A(T) \right\|^2_A \\
	 &\geq & \frac14\| T^\sharp T+TT^\sharp\|_A.
\end{eqnarray*}
\end{remark}

	\noindent \textbf{Declarations.} \\
The authors have no competing interests to declare that are relevant to the content of this article.

\bibliographystyle{amsplain}

\end{document}